\documentclass[11pt]{article}
\usepackage[pctex32]{graphics}
\usepackage{amssymb}
\usepackage{amsthm}
\usepackage{latexsym}
\usepackage{epsfig} 
\usepackage{multirow} 
\usepackage{pstricks}
\usepackage{float}
\usepackage{epstopdf}
\usepackage{mathrsfs}
\usepackage{soul}

\usepackage{amsmath}
\usepackage{amssymb}
\usepackage{relsize}
\usepackage{multirow}

\usepackage{graphicx}
\usepackage{multicol} 

\usepackage{times}
\usepackage{epstopdf}

\newtheorem{definition}{Definition}[section]
\newtheorem{theorem}[definition]{Theorem}
\newtheorem{lemma}[definition]{Lemma}

\definecolor{Red}{rgb}{1,0.,0.}

\usepackage{fancyhdr}
\usepackage{setspace}

\usepackage{hyperref}
\usepackage{mathtools}
\usepackage{algorithm}
\usepackage{algorithmic}
\usepackage{placeins}
\newcommand{\R}{{\mathbb R}}

\newcommand{\mS}{{\mathsf S}}

\newcommand{\mT}{{\mathsf T}}
\newcommand{\mA}{{\mathsf A}}

\newcommand{\mI}{{\mathsf I}}
\newcommand{\mC}{{\mathsf C}}
\newcommand{\mD}{{\mathsf D}}

\newcommand{\mQ}{{\mathsf Q}}
\newcommand{\mR}{{\mathsf R}}

\newcommand{\mV}{{\mathsf V}}
\newcommand{\mU}{{\mathsf U}}
\newcommand{\mM}{{\mathsf M}}
\newcommand{\mL}{{\mathsf L}}

\newcommand{\mSigma}{{\mathsf \Sigma}}
\newcommand{\mGamma}{{\mathsf \Gamma}}

\title{Distributed Tikhonov regularization for ill-posed inverse problems from a Bayesian perspective}
\author{D Calvetti \and E Somersalo}
\date{Department of Mathematics, Applied Mathematics, and Statistics\\
	Case Western Reserve University \\
	10900 Euclid Avenue \\
	Cleveland, OH 44106}

\begin{document}
\maketitle
\begin{abstract}

In this article, we exploit the similarities between Tikhonov regularization and Bayesian hierarchical models to propose a regularization scheme that acts like a distributed Tikhonov regularization where the amount of regularization varies from component to component, and a computationally efficient numerical scheme that is suitable for large-scale problems. In the standard formulation, Tikhonov regularization compensates for the inherent ill-conditioning of linear inverse problems by augmenting the data fidelity term measuring the mismatch between the data and the model output with a scaled penalty functional. The selection of the scaling of the penalty functional is the core problem in Tikhonov regularization. If an estimate of the amount of noise in the data is available, a popular way is to use the Morozov discrepancy principle, stating that the scaling parameter should be chosen so as to guarantees that the norm of the data fitting error is approximately equal to the norm of the noise in the data. A too small value of the regularization parameter would yield a solution that fits to the noise (too weak regularization) while a too large value would lead to an excessive penalization of the solution (too strong regularization). In many applications, it would be preferable to apply distributed regularization, replacing the regularization scalar by a vector valued parameter, so as to allow different regularization for different components of the unknown, or for groups of them. A distributed Tikhonov-inspired regularization is particularly well suited when the data have significantly different sensitivity to different components, or to promote sparsity of the solution. The numerical scheme that we propose, while exploiting the Bayesian interpretation of the inverse problem and identifying the Tikhonov regularization with the Maximum A Posteriori (MAP) estimation, requires no statistical tools. A clever combination of numerical linear algebra and numerical optimization tools makes the scheme computationally efficient and suitable for problems where the matrix is not explicitly available. 
\end{abstract}

\section{Introduction}

Consider the classical linear inverse problem of estimating an unknown vector $x\in\R^n$ from noisy indirect observations, described by the model
\begin{equation}\label{Ax=b}
 b = \mA  x+ \varepsilon, \quad \mA\in\R^{m\times n},
 \end{equation}
 where typically the matrix $\mA$ is ill-conditioned. In many applications where the data are expensive, difficult to collect, or the channels are limited by the instrumentation, the problem is underdetermined, i.e., $m<n$, and the presence of a non-trivial  null space ${\mathcal N}(\mA)$ exacerbates the ill-posedness of the problem because the solution, if it exists, is non-unique. In the following, when $m<n$, to guarantee the existence of the solution, we assume that ${\rm rank}(\mA) = m$. 
 
A way of addressing the ill-conditioning of the problem is to replace the problem by a nearby well-posed one, a process known as regularization \cite{engl1996regularization}.  One of the most popular approaches is Tikhonov regularization, which replaces (\ref{Ax=b}) with the minimization problem
 \[
  x_\alpha = {\rm argmin}\big\{\|\mA x - b\|^2 + \alpha \|x\|^2\big\},
 \]
 where $\alpha>0$ is the regularization parameter. If the norm of the noise $\varepsilon$ is given, e.g., in the deterministic form $\|\varepsilon\| \approx \delta$, the regularization parameter $\alpha$ can be selected according to the Morozov discrepancy principle stating that the value of $\alpha $ should be such that
 \[
  \|\mA x_\alpha - b\| = \delta.
 \]
 The idea behind Morozov discrepancy principle is that if the norm of the residual is below that of the noise, the solution is likely fitting to the noise. For other selection criteria, e.g., Generalized Cross Validation (GCV), Unbiased Predictive Risk Estimation (UPRE), L-curve method, we refer to \cite{engl1996regularization,hansen1998rank,hansen2010discrete,o1986statistical,kaipio2006statistical} Observe that by construction, $x_\alpha \perp {\mathcal N}(\mA)$, thus the Tikhonov regularized solution solves automatically the problem of non-uniqueness in the underdetermined case. A richer class of solutions can be obtained introducing a more general regularization term  
 \[
  x_{\mL,\alpha} = {\rm argmin}\big\{\|\mA x - b\|^2 + \alpha \|\mL x\|^2\big\},
 \]
 where $\mL\in\R^{k\times n}$ is matrix defining a seminorm. In this case the solution $ x_{\mL,\alpha}$ for any fixed parameter value $\alpha>0$ is unique provided that ${\mathcal N}(\mA)\cap{\mathcal N}(\mL) = \{0\}$.  The selection of the regularization parameter can be done by the same argument as before, however in general the regularized solution is no longer orthogonal to the null space of $\mA$, i.e., the matrix $\mL$ allows us to retrieve pertinent information from the null space not contained in the observation.

In many applications controlling the regularization level by one single regularization parameter is unsatisfactory. In geophysics \cite{li19963,li19983_2} and biomedical applications \cite{lin2006assessing,uutela1999visualization}, it is common to resort to some sort of distributed Tikhonov-type regularization scheme, minimizing an objective function of the form
\begin{equation}\label{lq}
 G(x) = \|\mA x - b\|^2 + \sum_{j=1}^n w_j |x_j|^q,
\end{equation}
where $1\leq q\leq 2$ and $w_j>0$. The non-uniform regularization is crucial because the sensitivity of the data to different components of the unknown may be vastly different, thus the components of less sensitivity need to be penalized less severely to avoid that the algorithm explains the data only in terms of the components with higher sensitivity. A popular rule of thumb is to introduce sensitivity weighting, choosing the weights as
\begin{equation}\label{sens}
 w_j = \alpha \| a^{(j)}\|, \quad a^{(j)} \in\R^m \mbox{ is the $j$th column vector of $\mA$,}
\end{equation}
while $\alpha$ is the Tikhonov regularization parameter. The scalar $q$ (\ref{lq}) is a parameter controlling the sparsity of the solution: Choosing $q=1$, the regularization, originating form geophysical applications \cite{santosa1986linear}, corresponds to the basis pursuit scheme \cite{chen1994basis,foucart2013invitation} in compressive sensing, or LASSO in the statistical literature \cite{tibshirani1996regression}.

In this work, we consider a regularization scheme in which each component of the vector $z = \mL x\in\R^k$ may contribute a different amount  to the solution. More precisely, the regularized solution is of the form 
\begin{equation}\label{general tikh}
 x_\theta = {\rm argmin}\left\{\|\mA x- b\|^2 + \sum_{\ell=1}^L \frac{\|\mL_\ell x\|^2}{\theta_\ell}\right\},
\end{equation}
with some matrices $\mL_\ell \in\R^{k_\ell\times n}$ and the reciprocals of the entries of the vector $\theta_\ell$ playing the role of distributed regularization functionals and parameters. There are several reasons to consider such generalizations of the Tikhonov regularization scheme: 
In addition to the applications with uneven sensitivities discussed above,
another important class of applications, elaborated further below, where distributed regularization is important is when the unknown to be estimated is sparse. Given a matrix $\mL\in\R^{k\times n}$,  we seek a solution $x$ such that
 \[
  z = \mL x, \quad \|z\|_0  = \#\{j \mid z_j \neq 0\}  \ll k,
 \]
 is {\em sparse} or more generally, for a given threshold $\delta>0$, $z$ is {\em compressible},
 \[
   \|z\|_{0,\delta}  = \#\{j \mid |z_j| > \delta\}  \ll k.
 \]  
 In the latter example, we call $\mL$ a  sparsifying operator.  Typically, we have $m<n<k$.  As will be pointed out in the discussion below, the sparsity property corresponds to a judicious choice of the prior assumptions about the parameter vector $\theta$.  The general form (\ref{general tikh}) allows also to consider group sparsity constraints that are important, e.g., in dictionary learning applications \cite{bocchinfuso2023bayesian}.

This article is a synthesis of a number of ideas put here in a common context in a unified manner, streamlining some computational ideas published earlier. Some of the computational ideas proposed in this article, while not new, seem not to have found their way in applications in which they may make a significant difference in speeding up the computations. Others, such as the analysis of the Tikhonov regularization parameter on the noise level when estimated by Bayesian methods, can prove to be important even for theoretical considerations of posterior consistency. 

\section{Hierarchical conditionally Gaussian prior model}

In this section, we introduce Bayesian hierarchical prior models and derive the associated sparsity-promoting optimization problem.  To set the notation, we denote by $x\in\R^n$ the variable of primary interest: adhering the practice in statistics of denoting random variables by capital letters and their realizations by lower case letters, let $X$ be the corresponding $n$-variate random variable. For a general reference of the background and methodology, see \cite{calvetti2023bayesian}. 

We start by introducing a general hierarchical model that covers several special models of interest.
 Let  $\mL\in\R^{k\times n}$, $k\geq n$, be a matrix of full rank, and assume that a partition $\Pi = \{I_\ell\}_{\ell=1}^k$ iof the index set $I=\{1,2,\ldots,n\}$ is given, that is,
 \[
  \{1,2,\ldots,n\} = \bigcup_{\ell=1}^L I_\ell, \quad I_j\cap I_\ell = \emptyset\mbox{ for $j\neq \ell$.}
 \] 
Define the corresponding row-partitioning of the matrix $\mL$ as
 \[
  \mL_\ell = \mL(I_\ell,\,:\,)\in\R^{k_\ell \times n}, \quad k_\ell = \# I_\ell,
 \] 
 with $k_1+\ldots + k_L = k$, and introduce auxiliary random variables $Z_\ell$ through
\[
 Z_\ell = \mL_\ell X \in\R^{m_\ell}, \quad 1\leq \ell\leq L.
\] 
We define a hierarchical prior model by postulating that the variables $Z_\ell$ are mutually independent zero-mean normally distributed random variables and that their covariance matrices are scaled identities. Further, we assume that the unknown scaling factors are mutually independent random variables following generalized gamma distributions, i.e.,
\begin{eqnarray*}
 Z_\ell  &\sim& {\mathcal N}(0,\theta_\ell\,\mI_{k_\ell}) \quad\mbox{mutually independent,} \\
 \Theta_\ell &\sim& {\rm Gen Gamma}(\beta_\ell,\vartheta_\ell, r)  \quad\mbox{mutually independent.}
\end{eqnarray*} 
Here  $\mI_{k_\ell}$ is the $k_\ell\times k_\ell$ identity matrix,  $\beta_\ell>0$  is the shape parameter, $ \vartheta_\ell>0$ is the scale parameter, and $r\neq 0$ is a parameter selecting the distribution type from the generalized gamma distribution family, determining the asymptotic behavior  as $\theta_\ell\to 0+$ and $\theta_\ell\to \infty$. We recall that the expression for the generalized gamma density function is 
 \[
  \pi_{\Theta_\ell}(\theta_\ell) = \frac{|r|}{\Gamma(\beta_\ell)\vartheta_\ell}\left(\frac{\theta_\ell}{\vartheta_\ell}\right)^{r\beta_\ell -1}
 {\rm exp}\left( - \left(\frac{\theta_\ell}{\vartheta_\ell}\right)^r\right).
\] 
The interpretation of the various hyperparameters will be discussed later. The heuristic justification for choosing a generalized gamma prior for the variances $\theta_\ell$ to promote sparsity is that they are positive fat-tailed distributions favoring relatively small values but allowing occasional large outliers. For a graphical rendition of the argument, see Figure~\ref{fig:hierarchical}.
\begin{figure}[ht!]
 \centerline{\includegraphics[width=9cm]{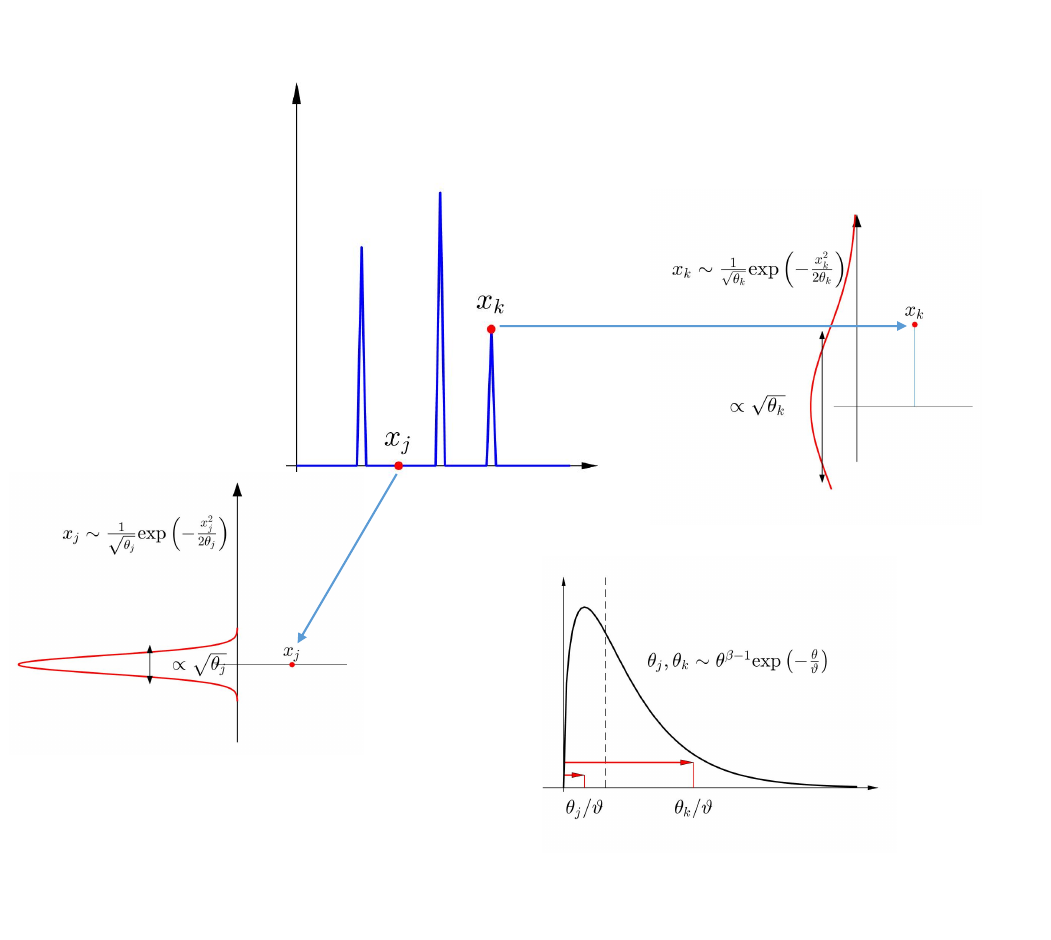}}
 \caption{\label{fig:hierarchical} A schematic picture justifying the selection of hyperpriors from the generalized gamma family. For simplicity, assume that $\mL=\mI$, and we require that $x$ is component-wise sparse. A component $x_j$ that is assumed to vanish or to be small is  a priori thought of being drawn from a Gaussian with small variance $\theta_j$, while a large component $x_k$ is drawn from a Gaussian with large $\theta_k$. A priori, we do not know where the large values are, but for sparsity, they should be rare, which is in line with the assumption that $\Theta$ has independent components with a heavy tailed distribution. }
 \end{figure}
 
 Introduce the vector
 \[
  Z = \left[\begin{array}{c} Z_1 \\ \vdots \\ Z_L\end{array}\right] \in \R^k,
 \] 
 and write the joint probability density of the pair $(Z,\Theta)\in\R^k\times \R^L$ as
\begin{eqnarray*}
&& \pi_{Z,\Theta}(x,\theta) = \pi_{Z\mid\Theta}(z\mid \theta)\pi_\Theta(\theta) \\
 && \phantom{X}\propto 
  {\rm exp}\left( - \frac 12 \sum_{\ell=1}^L \frac{\|z_\ell\|^2}{\theta_\ell} - \sum_{\ell=1}^L\left(\frac{\theta_\ell}{\vartheta_\ell}\right)^r +\sum_{\ell=1}^L\left(r\beta_\ell - \frac {k_\ell+2}2\right) \log\frac{\theta_\ell}{\vartheta_\ell}\right), \quad  z \in {\mathcal R}(\mL),
 \end{eqnarray*}
where ${\mathcal R}(\mL)\subset \R^k$ is the range of the matrix $\mL$. We remark that the density is concentrated on the subspace $ {\mathcal R}(\mL)\times\R^L$ because we assume that the variable $Z$ is related to a hidden variable $X\in\R^n$ through the relation $Z = \mL X$.
 
 Next, we introduce the likelihood model based on the forward model
 \[
  B = \mA X + E, \quad  E\sim{\mathcal N}(0,\mSigma),
 \]
where the noise covariance matrix $\mSigma\in\R^{m\times m}$ is symmetric positive definite, and let $\mSigma^{-1} = \mS^\mT\mS$ be a symmetric factorization for the precision matrix.
 Since a priori  $Z \in {\mathcal R}(\mL)$ and $\mL$ is assumed to have full rank, 
 \[
  Z = \mL X, \mbox{ hence $X = \mL^\dagger Z$,}
 \] 
with $\mL^\dagger$ denoting the pseudoinverse of $\mL$. Therefore the likelihood can be written in terms of the variable $Z$ as
 \[
  \pi_{B\mid Z} \propto {\rm exp}\left( -\frac 1{2}\big\|\mS\big( b - \mA \mL^\dagger z\big)\big\|^2\right).
 \] 
In the following, without loss of generality we assume that $\mS = \mI_m$, which is equivalent to replacing $b$ and $\mA$ by their whitened versions, $\mS b$ and $\mS\mA$, respectively.

It follows from Bayes' formula that the posterior density of $z$ is 
\begin{eqnarray*}
&&\pi_{Z,\Theta\mid B}(z,\theta\mid b) \propto \pi_{Z,\Theta}(z,\theta) \pi_{B\mid Z}(b\mid z) \\
&&\phantom{XX} \propto {\rm exp}\left( -\frac 1{2}\big\| b - \mA \mL^\dagger z\big\|^2  -  \frac 12 \sum_{\ell=1}^L \frac{\|z_\ell\|^2}{\theta_\ell}  -\Phi(\theta)\right), \quad z\in{\mathcal R}(\mL),
\end{eqnarray*}
where
\[
 \Phi(\theta) =\Phi_r(\theta,\beta,\vartheta) = \sum_{\ell=1}^L\left(\frac{\theta_\ell}{\vartheta_\ell}\right)^r - \sum_{\ell=1}^L\left(r\beta_\ell - \frac{k_\ell + 2}2\right) \log\frac{\theta_\ell}{\vartheta_\ell}.
\] 
In the following we summarize the posterior in terms of the Maximum a Posteriori (MAP) estimate, defined as
\begin{equation}\label{MAP}
 (z,\theta)_{\rm MAP} = \underset{(z,\theta)\in{\mathcal R}(\mL)\times \R^L}{\rm argmin}\left\{\frac 1{2}\big\| b - \mA \mL^\dagger z\big\|^2  + \frac 12 \sum_{\ell=1}^k \frac{\|z_j\|^2}{\theta_j}  + \Phi(\theta)\right\}, 
 \end{equation}
Next we establish the connection between the Bayesian hierarchical model and Tikhonov regularization in some special cases. 

Consider first a partition of the index set in singletons, $I_\ell = \{\ell\}$, $1\leq \ell \leq L = k$.This is tantamount to promoting component-wise sparsity, and the objective function to be minimized can be expressed in terms of $x = \mL^\dagger z$ as
  \begin{equation}\label{componentwise}
 G(x,\theta) = \frac 1{2}\big\| b - \mA x \big\|^2  + \frac 12 \sum_{\ell=1}^k \frac{\big[(\mL x)_\ell\big]^2}{\theta_\ell}  + \Phi(\theta).
\end{equation}
If $\mL$ is a finite difference matrix, promoting the sparsity of $\mL x$ is akin to total variation penalty, while letting $\mL = \mI_n$ and $k=n$ promotes a sparse solution of the linear inverse problem.

At the other extreme, we consider a trivial partition $\Pi = \{ I\}$, in which case the objective function reduces to
 \begin{equation}\label{Tikhonov}
 G(x,\theta) = \frac 1{2}\big\| b - \mA x \big\|^2  + \frac 12 \frac{\|\mL x\|^2}{\theta}  + \Phi(\theta),
\end{equation}
from which we deduce that this is the Bayesian equivalent of the whitened standard Tikhonov regularization with regularization parameter $\alpha = 1/\theta$. The optimization algorithm for the MAP computation is the Bayesian way of determining the value of the regularization parameter, similarly to using Morozov's discrepancy principle or any other alternative.

Finally, consider an example in which the matrix $\mL$ is block diagonal,
\[
 \mL = \left[\begin{array}{cccc} \mGamma_1 &  &  &  \\  & \mGamma_2 & & \\
  &  & \ddots & \\ & & & \mGamma_k\end{array}\right] = \left[\begin{array}{c} \mL_1 \\ \mL_2 \\ \vdots \\ \mL_k\end{array}\right], \quad \mGamma_\ell \in \R^{k_\ell\times k_\ell},
\] 
and the matrices $\Gamma_\ell$ are symmetric positive definite.
In this case, the objective function becomes
  \begin{equation}\label{structural}
 G(x,\theta) = \frac 1{2}\big\| b - \mA x \big\|^2  + \frac 12 \sum_{\ell=1}^k \frac{\|\mGamma_\ell^{1/2}  x^{(\ell)}\|^2}{\theta_\ell}  + \Phi(\theta),
\end{equation}
where $x^{(\ell)} = x(I_\ell)$, that is, the unknown is partitioned into $k$ blocks. In this case the sparsity is promoted on the groups of components, and the matrices $\mGamma_\ell^{-1}$ account for the intrinsic covariance of the $\ell$th block of components. Example of prior assumptions of this kind can be found in \cite{calvetti2015hierarchical,bocchinfuso2023bayesian}.

\section{Computing the MAP estimate}

In this section we review an algorithm, known as the Iterative Alternating Sequential (IAS) algorithm, for computing the MAP estimate for the hierarchical model introduced earlier, which can be regarded as the Bayesian generalization of the Tikhonov regularized solution. The computation of the minimizer of the objective function is achieved by a block-descent scheme which updates sequentially the unknown of primary interest by solving a Tikhonov regularization problem in standard form and the parameters of the prior covariance. We start by reviewing the optimization scheme for solving the minimization problem (\ref{MAP}). 

\bigskip
\hrule
\medskip

{\bf Iterative Alternating Sequential (IAS)} algorithm
\medskip
\hrule
\medskip

\begin{enumerate}
\item {\bf Given} the forward map $\mA\in\R^{m\times n}$, input data $b\in\R^m$,  full rank matrix $\mL\in\R^{k\times n}$ with a partitioning $\Pi$, parameters $(r,\beta,\vartheta)$,  tolerance $0<\delta<1$.
\item {\bf Initialize} $\theta^0 = \vartheta$, and set $t=0$.
\item {\bf Iterate}: 
\begin{enumerate}
\item {\bf Phase I:} Update $z\in\R^k$ by minimizing 
\[
G_{\theta^t}(z) = \frac 1{2}\big\| b - \mA \mL^\dagger z\big\|^2  + \frac 12 \sum_{\ell=1}^L \frac{\|z_\ell\|^2}{\theta^t_\ell} , \quad z \in {\mathcal R}(\mL),
\]
\item {\bf Phase II:} Update $\theta\in\R^L$ by minimizing
\[
 G_{z^{t+1}}(\theta) =  \frac 12 \sum_{\ell=1}^L \frac{\|z^{t+1}_\ell\|^2}{\theta_\ell}  + \Phi(\theta) ;
\] 
\end{enumerate}
\item If $\|\theta^{t+1} - \theta^t\|/\|\theta^t\| < \delta$, stop, else advance the counter $t\rightarrow t+1$ and continue from 3.
\end{enumerate}

\medskip
\hrule
\bigskip

Below, after presenting the details of  the two phases of the algorithm, we discuss questions of its convergence and present some interpretations in the light of the similarities with Tikhonov regularization. 

\subsection{Phase I: Updating z}\label{sec:update z}

The first phase in each IAS iteration is the updating of $z$. Defining the $k\times k$ diagonal matrix,
\[
 \mD_\theta = \left[\begin{array}{cccc} \theta_1\,\mI_{k_1} &&& \\ &\theta_2 \, \mI_{k_2} && \\ &&\ddots & \\ &&& \theta_L\,\mI_{k_\ell}\end{array}\right],
\]
and recalling that $z = \mL x \in{\mathcal R}(\mL)$,  we can write the second term of $G_{\theta^t}(z)$ as
\[
 \sum_{\ell=1}^L \frac{\|z_\ell\|^2}{\theta_\ell} = \| \mD_\theta^{-1/2} \mL x\||^2.
\]
After performing the change of variables
\[
 \xi =  \mL_\theta x, \quad \mL_\theta = \mD_\theta^{-1/2} \mL,  \quad x = \mL_\theta^\dagger \xi, \quad \theta = \theta^{t+1},
\]
the objective function can be written as
\[
 \widetilde G_\theta(\xi) = \frac 12 \|b - \mA \mL_\theta^\dagger \xi\|^2 + \frac 12 \|\xi\|^2, \quad \xi\in {\mathcal R}(\mL_\theta),
\]  
and since the first term does not depend on the orthogonal projection of $\xi$ onto the null space of $\mL_\theta^\mT$, the minimizer of $\widetilde{G}_\theta$ is automatically in the subspace ${\mathcal R}(\mL_\theta)$, hence there is no need to impose the constraint. Therefore the updating of $z$ essentially amounts to solving a Tikhonov-type least squares problem of standard form,
\begin{equation}\label{get xi}
 \|b - \mA\mL_\theta^\dagger \xi\|^2 + \|\xi\|^2  =  \left\|\left[\begin{array}{c} \mA\mL_\theta^\dagger \\ \mI\end{array}\right] \xi - \left[\begin{array}{c} b \\ 0\end{array}\right]\right\|^2.
\end{equation}
While in principle this is a very straightforward problem, it may become a challenge when the problem is high dimensional or the matrix $\mA$ is not available in explicit form. We will discuss later how these issues can be addressed.

\subsection{Phase II: Updating the variances}

By construction, the components $\theta_\ell$ are mutually independent, therefore the objective function to be minimized when updating $\theta$ is of the form
\[
 G_z(\theta) = \sum_{\ell=1}^L g_\ell(\theta_\ell), \quad g_\ell(\theta_\ell) = \frac 12 \frac{\|z_\ell\|^2}{\theta_\ell}  + \left(\frac{\theta_\ell}{\vartheta_\ell}\right)^r - \left(r\beta_\ell - \frac{k_\ell + 2}2\right)\log\frac{\theta_\ell}{\vartheta_\ell}
\] 
After the non-dimensionalization step
\[
 t_\ell  = \frac{\|z_\ell\|}{\sqrt{\vartheta_\ell}}, \quad \lambda_\ell = \frac{\theta_\ell}{\vartheta_\ell},\quad \eta_\ell = r\beta_\ell - \frac{k_\ell + 2}2,
\] 
the component-wise objective function assumes the form
\[
\widetilde g_\ell (\lambda_\ell) =  \frac 12\frac{t_\ell^2}{\lambda_\ell} + \lambda_\ell^r - \eta_\ell \log\lambda_\ell,
\] 
and its minimum is attained at a critical point satisfying the first order optimality condition
\begin{equation}\label{gprime}
 \widetilde g'_\ell(\lambda_\ell) = -\frac{t_\ell^2}{2\lambda_\ell^2} + r\frac{\lambda_\ell^r}{\lambda_\ell} -\frac{\eta_\ell}{\lambda_\ell} = 0.
\end{equation}
For some values of $r$, this equation has explicit closed form solution, for example when $r=1$, assuming that $\eta_\ell>0$, corresponding to the gamma distribution for the variances, the non-negative solution is
\[
 \lambda _\ell = \frac{1}{2}\left(\eta_\ell + \sqrt{\eta_\ell^2 + 2 t_\ell^2}\right).
\] 
When $r=-1$, corresponding to inverse gamma distribution, we have $\eta_\ell<0$ and obtain
\[
 \lambda_\ell = \frac{1}{|\eta_\ell|}\left( \frac{t_\ell^2}{2}+1 \right).
\] 
For other values of $r$, the numerical solution can be found by denoting by $\varphi_\ell$ the functional dependency of the solution $\lambda_\ell$ on $t_\ell$. After substituting the expression $\lambda_\ell = \varphi_\ell(t_\ell)$ in (\ref{gprime}), differentiating implicitly with respect to $t_\ell$, and solving for the derivative of $\varphi_\ell$, we find that the function $\varphi_\ell$ satisfies the initial value problem
\[
  \varphi_\ell'(t) = \frac{2 t\varphi_\ell(t)}{2r^2\varphi_\ell(t)^{r+1} + t^2},\quad \varphi_\ell(0) = \left(\beta_\ell - \frac{k_\ell + 2}{2r}\right)^{1/r},
\] 
where we need to set the parameters so that $\beta_\ell - (k_\ell+2)/(2r)>0$. Therefore we can find $\lambda_\ell$ by solving the initial value problem by a standard Runge-Kutta scheme.

\section{Convergence and parameter selection in IAS Algorithm}

The convergence properties of the IAS algorithm have been investigated in the literature. In \cite{calvetti2019hierachical}, the following theorem was proved for the case when $\mL = \mI_n$ and the partitioning of the matrix is $\Pi = \big\{\{\ell\}\big\}_{\ell=1}^n$, i.e., component-wise sparsity is promoted. In the theorem, we assume that the hyperparameters $\beta_\ell = \beta$ are equal for all $\ell$, $\leq \ell\leq L=n$ and $\eta = \beta - 3/2 >0$.

 \begin{theorem}
The IAS algorithm with gamma hyperprior ($r=1$) converges to the unique minimizer $(\widehat{z},\widehat{\theta})$ of the Gibbs energy functional. Moreover, the minimizer $(\widehat{z},\widehat{\theta})$ satisfies the fixed point condition
\[ \widehat{z} = \underset{z}{{\rm argmin}} \left\{ {\mathscr E}\big( z\mid f(z)\big) \right\}, \quad \widehat{\theta}= f(\widehat{z}),
\]
where $f$ is the map with $j$th component 
\[
 f_j(z_j) = \vartheta_j\left(\frac\eta 2 + \sqrt{\frac{\eta^2}{4} + \frac{2 z_j^2}{\vartheta_j}}\right). 
\]  
\end{theorem}

The theorem also holds for other regularization operators $\mL$ and its partitionings. One such generalization  was given in \cite{calvetti2015hierarchical} where the inverse problem of magnetoencephalography (MEG) with an anatomical group sparsity prior of the type (\ref{structural}) was considered.

The following theorem from \cite{calvetti2019hierachical} elucidates  the role of different prior parameters. As in the previous theorem, we assume a component-wise partitioning and set the shape parameter $\eta = \beta - 3/2>0$ for all components.

\begin{theorem}
When $\eta\to 0+$, the MAP estimate $\widehat z = \widehat z_\eta$ converges to the minimizer of the functional
\begin{equation}\label{ell1}
  G_0( z)
=  \frac 12 \|b -\mA z\|^2 + 
  {\sqrt{2}} \sum_{j=1}^n \frac{|z_j|}{\sqrt{\vartheta_j}}.
 \end{equation}
\end{theorem}

This theorem highlights the connection between sparsity-promoting Bayesian hierarchical models and classical $\ell_1$-based sparsity-promoting Tikhonov-type penalties. Moreover, it indicates how the value of the shape parameter should be set: the smaller the value of $\eta>0$, the more strongly the algorithms proposes sparsity. The theorem also shows that the scaling parameters $\vartheta_j$ are related to sensitivity weighting of the components. While the sensitivity weighting scaling (\ref{sens}) can be justified heuristically, its Bayesian interpretation is less straightforward, as it amounts to assigning higher a priori preference for components with lower sensitivity. In \cite{calvetti2019brain}, a Bayesian justification for sensitivity weighting was given in terms of the signal-to-noise ratio (SNR). In the cited work, the concept of SNR-exchangeability was introduced, stating that any subset of components with equal cardinality should have equal probability to explain the data of given SNR level. The SNR  exchangeability  principle automatically leads to a selection rule of the hyperparameters $\vartheta_j$ as
\[
 \vartheta_j = \alpha \frac{{\rm SNR}}{\|a^{(j)}\|^2},
\]
where ${\rm SNR}$ is the estimated signal-to-noise ratio, $a^{(j)}$ is as in (\ref{sens}), and $\alpha$ is a parameter depending on the belief of the cardinality of the solution. The important observation here is that in light of (\ref{ell1}), such choice automatically reproduces the sensitivity  weighting. 

I
Compared to the Tikhonov regularization with Morozov discrepancy principle, in Bayesian hierarchical model, slightly more information than the noise level is required, namely an estimate of the SNR. However, testing of the viability of the Morozov discrepancy principle, it is common to express the noise level in terms of percentages of the noiseless signal, hence it may be argued that the notion of SNR is implicitly built in the method. The connection will be elucidated further in the computed experiments. 
 
 \subsection{Non-convexity and hybrid IAS}
 
The uniqueness and global convergence of the IAS algorithm for the hierarchical model have been established only  for the case $r=1$, when the hyperprior is the gamma distribution. The convexity properties of the objective function for $r \neq 1$ have been analyzed in detail  in  \cite{calvetti2020sparse} where it is shown that for $r < 1$, the objective function in general is not globally convex, but local convexity regions can be identified. Moreover, inside the convex region, faster convergence towards the local minimizer results into very efficient cleaning of the background. Motivated by these observations, a hybrid algorithm was suggested and tested in \cite{calvetti2020sparsity}. The hybrid hierarchical model scheme proceeds with $r=1$ until the MAP estimate is near the global minimum of the globally convex objective function, and subsequently switches to a greedier hyperprior corresponding to with $r<1$ converging fast to a possibly local minimum.

An important question in the hybrid scheme is how to set the hyperprior parameters to guarantee consistency of the two schemes. A compatibility condition, initially proposed in \cite{calvetti2020sparsity} and later rivisited in \cite{calvetti2023computationally}, based on the following two conditions:
\begin{enumerate}
\item When $z_j=0$, the baseline values for $\theta_j$ computed by the two models should coincide, in agreement with the understanding that the a priori variance of the components outside the support of $z$ is independently of the model. 
\item The marginal expected value for the variances should be the same for both models.
\end{enumerate}
  
Denoting by $(r_1,\beta_1,\vartheta_1)$ and $(r_2,\beta_2,\vartheta_2)$ the hyperparameters for the two models employed in the hybrid scheme, with $r_1=1$, we can set the hyperparameters of the second model so as to be in agreement with the compatibility conditions as follows.
\begin{enumerate}
\item The statement on the consistency of the variance of the components in the complement of the support lead to the equation
\begin{equation}\label{compatibility}
 \vartheta_{1,\ell}\left(\beta_{1,\ell} - \frac{k_\ell+2}{2r_1}\right)^{1/r_1} =  \vartheta_{2,\ell}\left(\beta_{2,\ell} - \frac{k_\ell+2}{2r_2}\right)^{1/r_2}.
\end{equation}
\item Recalling the expression for expectation of generalized gamma distributions, the condition on the marginal expected value for the variances implies that
\begin{equation}\label{compatibility2}
\vartheta_{1,\ell}\frac{\Gamma(\beta_{1,\ell} + \frac 1{r_1})}{\Gamma(\beta_{1,\ell})} = \vartheta_{2,\ell}\frac{\Gamma(\beta_{2,\ell} + \frac 1{r_2})}{\Gamma(\beta_{2,\ell})}.
\end{equation} 
\end{enumerate}
Additional conditions on the values of the hyperparameters may be needed to guarantee positivity of some variable combinations as well as finite expectations of the probability density functions. Using the properties of the gamma function, in \cite{calvetti2023computationally}, the formulas for setting the parameters $\vartheta_2$ and $\beta_2$ for the special case $k_\ell = 1$ were worked out in detail for the values $r_2\in\{1/2,-1/2,-1\}$.

\section{Computationally efficient algorithms }

The computational challenges for updating the unknown of interest $z$ as discussed in section~\ref{sec:update z} may be significant when the unknown to be estimated is high dimensional, the computing resources are limited, the problem needs to be solved in real time, or if the forward model matrix $\mA$ is not available, as is the case in matrix-free formulations of the forward model.

\subsection{Computational Complexity}
In problems where the unknown is significantly higher dimensional than the data, i.e., $m\ll n$,  it is possible to reduce the computational complexity of Phase I of the IAS algorithm by taking advantage of the low dimensionality of the data.
In Phase I of each IAS iteration for the update of $z$ it is necessary to solve (\ref{get xi}), which is equivalent to computing 
\begin{equation}\label{LSxi}
\xi = \underset{\xi \in\R^n}{\rm{argmin} }\left\|\left[\begin{array}{c} \mA\mL_\theta^\dagger \\ \mI \end{array}\right] \xi - \left[\begin{array}{c} b \\ 0\end{array}\right]\right\|^2.
\end{equation}
Let 
\[
\mA_\theta = \mA\mL_\theta^\dagger.
\]
The following well-known theorem proves that if $m<n$, the vector $\xi$ satisfying (\ref{LSxi}) can be obtained by solving a lower dimensional adjoint problem \cite{calvetti2000regularizing}.
\begin{lemma}\label{TikWien}
Let $\mA_\theta \in \R^{m \times n}$ and $m<n$, and let $\zeta\in\R^m$ be the  unique solution of the $m \times m$ linear system
\begin{equation}\label{WienEq}
\big( \mA_\theta  \mA_\theta^\mT+\mI_m \big) \zeta =  b.
\end{equation}
Then 
\[
\xi = \mA_\theta^\mT \zeta\in\R^n
\]
is the unique solution of the $n \times n$ linear system
\begin{equation}\label{TikEq}
\big(\mA_\theta^\mT \mA_\theta +\mI_n \big) \xi = \mA_\theta^\mT b.
\end{equation} 
\end{lemma}
\begin{proof}
Multiply (\ref{WienEq}) from the left by $\mA_\theta^\mT $ to get 
\begin{eqnarray*}
 \mA_\theta^\mT b &=& \mA_\theta^\mT\big( \mA_\theta  \mA_\theta^\mT+\mI_m \big) \zeta \\
 &=& \big( \big(\mA_\theta^\mT \mA_\theta \big) \mA_\theta^\mT + \mI_n\mA_\theta^\mT \big) \zeta \\
 &=& \big( \mA_\theta^\mT  \mA_\theta+\mI_n \big) \mA_\theta^\mT \zeta,
\end{eqnarray*}
therefore $\mA_\theta^\mT \zeta$ is the unique solution of (\ref{TikEq}).
\end{proof}
In the remainder of the paper we will refer to (\ref{TikEq}) as Tikhonov equations and to the adjoint problem (\ref{WienEq})  as Wiener equations. It follows from Lemma~\ref{TikWien} that mathematically the update of $z$ in Phase I of the IAS algorithm can be done using either the (\ref{TikEq}) or the (\ref{WienEq}) formulation. Computationally, when $m<n$ it may be preferable to compute first the solution of the Wiener equations because it will reduced the computational complexity. The more severely underdetermined the problem, the higher the complexity reduction. 

\subsection{ Krylov-Lanczos least squares solution}   

The numerical solution of linear inverse problems must address the fact that all computations are performed in finite precision arithmetic, and that the memory allocation may not allow the storage of either  $\mA_\theta  \mA_\theta^\mT$ or  $\mA_\theta^\mT  \mA_\theta$. Moreover, in many applications the matrix $\mA_\theta$ is not explicitly given, but it is possible to multiply a vector by it or its transpose. 
Iterative linear solvers,  an alternative to direct solution methods based on a factorization of the coefficient matrix that only access the matrix or its transpose to perform matrix-vector products, have long been the methods of choice for large or matrix-free problems. 

The literature of iterative linear solvers is very vast. In this paper we limit our attention to Krylov subspace iterative solvers. Given a square linear system 
\begin{equation}\label{KrylovLS}
\mM x = c, 
\end{equation}
the $k$th Krylov subspace associated with $\mM$ and $c$ is, by definition,
\[
\mathcal{K}_k(\mM,c) = \rm{span} \{ c, \mM c, \ldots, \mM^{k-1}c\}, \quad k=1,2,\ldots.
\]
Given an initial approximate solution $x_0$ and corresponding initial residual error $r_0= c - \mM x_0$, a Krylov subspace iterative solvers for (\ref{KrylovLS}) determines a sequence of approximate solutions $x_k \in \mathcal{K}_k(\mM,r_0)$ that minimize a given functional in the Krylov subspace. Different iterative solvers are characterized by the objective function. One of the most popular Krylov subspace iterative solver for symmetric positive definite systems is the Conjugate Gradient (CG) method, introduced in 1952 \cite{hestenes1952methods} as an alternative to Gaussian elimination. At each iteration, the approximate solution determined by the CG method minimizes the energy norm of solution the error, i.e.,
\[
x_k = \underset{x\in \mathcal{K}_k(\mM,c)} {\rm argmin} \| x -  x_*\|_\mM^2 , \quad \| x -  x_*\|_\mM^2  = (x-x^*)^\mT \mM (x-x^*), \quad x_*= \mM^{-1}c.
\]
The seminal 1952 paper by Hestenes and Stiefel  also proposes a variant of the conjugate gradient method for linear systems (\ref{KrylovLS}) that arise as the normal equations associated with linear least squares problems $\mA x = b$, i.e., $\mM = \mA^\mT \mA$, and $c = \mA^\mT b$. The numerical stability of algorithms implementing CG for least squares problems is addressed in \cite{paige1982lsqr}; see also \cite{bjorck1996numerical}. Since the functional minimized by the iterates of the Conjugate Gradient for Least Squares (CGLS) iterative solver over a nested sequence of Krylov subspaces is the Euclidean norm of the discrepancy, $\| b - \mA x\|$, the CGLS method has been used extensively in the iterative solution of linear inverse problems using the iteration index as a regularization tool \cite{hanke1993regularization,hanke2017conjugate}. 

In the same year as the CG method was proposed, the question how to solve real symmetric linear systems by minimized iterations was being addressed from a different point of view by Lanczos \cite{lanczos1952solution}, who approached it by introducing families of orthogonal polynomials with respect to an inner product defined by the coefficient matrix and the right-hand side vector; see \cite{golub2009matrices} for an overview and extensive bibliography.

The algorithm proposed by Lanczos for determining an orthonormal basis of Krylov subspaces only accessing the matrix to form matrix-vector products, uses the three-term recurrence relations of the underlying orthogonal polynomials, thus determining an orthogonal tridiagonalization of the coefficient matrix.  Equivalently, at the $k$th iteration the Lanczos process determines an orthonormal basis of the $k$th  Krylov subspace and the tridiagonal projection of the restriction of the coefficient matrix to the subspace. It can be shown that the eigenvalues of the tridiagonal matrix are approximations of the eigenvalues of the original matrix, therefore the Lanczos process is often used for numerical studies of eigenvalue problems. More specifically, given the linear system 
\[
 \mM y = c, \quad \mbox{$\mM\in\R^{m\times m}$ symmetric.}
\]
$\ell$ steps of Lanczos process determines a set of orthonormal vectors $v_1,v_2,\ldots,v_{\ell}$ such that
\[
 {\mathcal K}_\ell(c,\mM)  = {\rm span}\big\{v_1,v_2,\ldots,v_{\ell}\big\}.
\]

The organization of the computations is summarized in the following algorith,

\bigskip
\hrule
\medskip

{\bf Lanczos tridiagonalization algorithm}

\medskip
\hrule
\medskip

 \begin{enumerate}
 \item {\bf Given} $\mM\in\R^{m\times m}$ symmetric, $0\neq c\in\R^m$.
 \item {\bf Initialize} $v_1 = c/\|c\|, \; \gamma_0=0; \; v_0 = v_1 $.
 \item {\bf Repeat} for $j = 1, 2,\ldots,\ell$:
 \begin{enumerate}
 \item $w = \mM v_j$,
 \item $w = w - \gamma_{j-1} v_{j-1}$, 
\item $\alpha_{j-1} = v_j^\mT w$,
\item $w = w - \alpha_{j-1} v_j$,
\item $\gamma_j = \|w\|$,
\item if $\gamma_j=0$, exit, else $v_{j+1} = w/\gamma_j$.
 \end{enumerate}
 \end{enumerate}
 
\medskip
\hrule
\medskip

The output of $\ell$ steps of the Lanczos tridiagonalization algorithm is the $\ell+1$ orthonormal vectors $v_1, \ldots, v_{\ell+1}$, the scalars $\alpha_1, \ldots, \alpha_\ell$ and the positive scalars $\gamma_1, \ldots, \gamma_\ell$.  The matrices 
\[
\mV_{\ell +1} = [\begin{array}{c c c} v_1& \ldots & v_{\ell+1}  \end{array}]  \in\R^{m\times (\ell+1)}, \quad 
\mT_{\ell+1, \ell} = \left[ \begin{array}{ c c c c c} \alpha_0 & \gamma_1 & &  &\\  \gamma_1 & \alpha_1 & \gamma_2 & & \\ & \ddots & \ddots & \ddots & \\
& &\gamma_{\ell-2}& \alpha_{\ell-2} & \gamma_{\ell-1}    \\ &  & & \gamma_{\ell-1} & \alpha_{\ell-1} \\
 & & & &\gamma_\ell
\end{array}\right] \in\R^{(\ell +1) \times \ell},
\]
satisfy  the identity
\begin{equation}\label{BV1}
\mM \mV_\ell  = \mV_{\ell+1}\mT_{\ell+1,\ell} 
 = \mV_{\ell} \mT_{\ell} + \gamma_\ell v_{\ell+1} e_\ell^\mT,
\end{equation}
holds, where $\mT_\ell$ is the $\ell\times\ell$ block of first $\ell$ rows of $\mT_{\ell,\ell+1}$ and $e_\ell\in\R^m$ is the canonical $\ell$th unit vector. If $\gamma_\ell=0$ then
\[
\mM \mV_\ell = \mV_\ell \mT_\ell,
\]
hence the eigenvalues of $\mT_\ell$ are eigenvalues of $\mM$.

Since the Lanczos vectors $v_1, \ldots, v_{\ell}$ are an orthonormal basis of $\mathcal{K}_\ell (c,\mM)$,  every vector $y \in \mathcal{K}_\ell (c,\mM)$ can be written as $y = \mV_\ell z$ for some $z \in \R^\ell$. 
In particular, the $\ell$th approximate solution   $y^{(\ell)}$ of   
\[
 \big(\mM + \mI_m) y = c, 
\]
satisfying 
\begin{equation}\label{y_ell}
y^{(\ell)} = \underset{y \in {\mathcal{K}}_\ell(c, \mM)}{\mathrm{argmin}}
 \| (\mM + \mI_m) y -c \|^2
\end{equation}
must be of the form $\mV_\ell z^{(\ell)}$. We observe that in the light of (\ref{BV1}) the quantity to be minimized can be written as
\begin{eqnarray*}
 \| (\mM + \mI_m) \mV_\ell z - c \|^2  &=& \big\| \mV_\ell  (\mT_\ell  +\mI_\ell) z + \gamma_\ell v_{\ell+1} z_\ell  -\|c\|v_1\big\|^2 \\
                         & = & \big\| (\mT_\ell+ \mI_\ell) z  - \|c\| e_1 \big\|^2  + \gamma_\ell^2 z_\ell^2
 \end{eqnarray*}
where the last identity follows from the orthonormality of the columns of $\mV_{\ell+1}$. Then $y^{(\ell)} = \mV_\ell z^{(\ell)}$, with  
 \begin{equation}\label{z_ell}
  z^{(\ell)} = \underset{z\in\R^\ell}{\mathrm{argmin}}  \big\| (\mT_{\ell+1,\ell}+ \mI_{\ell+1,\ell}) z  - \|c\| e_1 \big\|^2, 
 \end{equation}
 where the matrix $\mI_{\ell+1,\ell}$ consists of the first $\ell$ columns of the $(\ell+1) \times (\ell+1)$ identity $\mI_{\ell+1}$ and $e_1$ is the first  column of $\mI_{\ell+1}$. 
 We remark that if $\gamma_\ell = 0$ the Lanczos process terminates: this is guaranteed to happen when $\ell = m$, but it may be reached for smaller values of $\ell$. 

Finally, notice that for a given vector $c$, the Lanczos vectors remain invariant under addition of a scalar multiple of the identity to the matrix $\mM$. For this reason, the Lanczos process has been used extensively to determine the value of the regularization parameters in the context of Tikhonov regularization; see, e.g., \cite{calvetti2000tikhonov, chung2021computational}.

The orthogonality of the Lanczos vectors may not be guaranteed in finite precision arithmetics, due to the accumulation of roundoff errors, sometimes making their reorthogonalization necessary.  Furthermore, it has been observed \cite{paige1982lsqr} that when the symmetric matrix $\mM$ is of the form $\mM = \mA \mA^\mT$ or $\mM = \mA^\mT \mA$, the orthogonality of the Lanczos vector can be preserved better by working separately with the matrices $\mA$ and $\mA^\mT$ rather than forming their product. The resulting algorithm, known as Lanczos bidiagonalization, determines two sets of orthonormal vectors and a bidiagonal matrix. We refer to \cite{paige1982lsqr} for the details on how to organize the Lanczos bidiagonalization for the solution of least squares problems of the type (\ref{TikEq}) and (\ref{WienEq}). 
For completeness, we state the Lanczos bidagonalization algorithm suitable for (\ref{WienEq}).

\bigskip
\hrule
\medskip

{\bf Lanczos bidiagonalization algorithm}

\medskip
\hrule
\medskip

 \begin{enumerate}
 \item {\bf Given} $\mA \in\R^{m\times n}, \quad 0\neq b\in\R^m,\; 1<\ell <\min(m,n);$
 \item {\bf Initialize} $\sigma_1= \|b\|;\; v_1= b/\sigma_1; \;  \tilde{u} =\mA^\mT v_1; \; \rho_1= \| \tilde{u}\|; \; u_1 = \tilde{u}/\rho_1$.
 \item {\bf Repeat} for $j = 1,2,\ldots,\ell$
 \begin{enumerate}
 \item $\tilde{v}_j = \mA u_{j-1} - \rho_{j-1}v_{j-1}, \quad \sigma_j = \|\tilde{v}_j, \quad v_j = \tilde{v}_j/\sigma_j$;
 \item $\tilde{u}_j = \mA^\mT v_j - \sigma_j u_{j-1}, \quad \sigma_j = \|\tilde{u}_j \|, \quad u_j = \tilde{u}_j/\sigma_j$; 
 \end{enumerate}
 \item {\bf Compute}  $\tilde{v}_{\ell+1} = \mA u_\ell - \rho_\ell v_\ell,$ 
 \quad $\sigma_{\ell+1}= \|\tilde{v}_{\ell+1}\|, $
 \quad $v_{\ell+1}= \tilde{v}_{\ell+1} /\sigma_{\ell+1}$.
 \end{enumerate}

\medskip
\hrule
\medskip

One can show that, by construction,  the algorithm produces matrices
\[
\mU_\ell =[ \begin{array}{ c c c} u_1 & \ldots & u_\ell \end{array} ], \quad \mV_\ell =[ \begin{array}{ c c c} v_1 & \ldots & v_\ell \end{array} ],
\]
and
\[
\mC_\ell = \left[ \begin{array}{ c c c c c} \rho_1 &  & &  &\\ \sigma_1 & \rho_2 & & & \\ & \ddots & \ddots & &  \\ & & \sigma_{\ell-1} & \rho_{\ell-1} \\ &&&\sigma_\ell & \rho_\ell
\end{array}\right].
\]
satisfying the identities
\begin{eqnarray*}
\mA \mU_\ell &= &\mV_\ell \mC_\ell + \sigma_{\ell+1}v_{\ell+1} e_\ell^\mT\\
 \mA^\mT \mV_\ell &= &\mU_\ell \mC_\ell^\mT.
\end{eqnarray*}
In particular, we observe that if $\mM = \mA\mA^\mT$, 
\begin{eqnarray}\label{BV2}
 \mM \mV_\ell&=& \mA \mA^\mT \mV_\ell   \nonumber\\
 &=&  \mA \mU_\ell \mC_\ell^\mT  \nonumber  \\
 &=&V_\ell \mC_\ell\mC_\ell^\mT + \sigma_{\ell+1}v_{\ell+1} \big(\mC_\ell e_\ell\big)^\mT \nonumber \\
  &=&V_\ell \mC_\ell\mC_\ell^\mT + \sigma_{\ell+1} \rho_\ell v_{\ell+1} e_\ell^\mT.
\end{eqnarray}
By comparing the identities (\ref{BV1}) and (\ref{BV2}) we observe that the  Lanczos bidiagonalization algorithm produces the Cholesky factorization of the tridiagonal matrix $\mT_\ell$. The $\ell$th approximate solution $y^{(\ell)}$ can be computed from the output of the Lanczos 
bidiagonalization in a manner similar to that described for the Lanczos tridiagonalization, hence we omit the details.  

A natural question that arises is how many steps of the Lanczos process should be run. Ideally, in order to compute update of  $z$ exactly we need to run the Lanczos process until (\ref{TikEq}) or  (\ref{WienEq}) are solved exactly. It can be shown that the last subdiagonal entry of the   Lanczos tridiagonal matrix $\sigma_\ell$ is equal to the norm of the residual error of the adjoint equations
\[
r_\ell = b - (\mA \mA^\mT + \mI_m) \mV_\ell y^{(\ell)},
\]
therefore the process can continued until the adjoint problem has been solved to a desirable accuracy.  

We conclude this section with the observation that if computer memory is limited, it is possible to compute the $\ell$th approximate solution of the Lanczos process storing only two orthogonal vectors and the nonzero entries of the tridiagonal, or bidiagonal matrix. The cost of keeping the memory free is that once enough steps of the process have been completed and the vector $z^{(\ell)}$ has been computed, the solution vector is formed by adding together the vectors $v_j$ scaled by the corresponding component of $z^{(\ell)}$. In this case, however, there is no way to overcome the potential loss of orthogonality of the Lanczos vectors with reorthogonalization.

\subsection{Computation of the pseudoinverse of L}\label{sec:qr}

The updating  of the variable $\xi$ in the IAS algorithm is based on the use of the pseudoinverse of the matrix $\mL_\theta$ that needs to be recomputed every time $\theta$ is updated. If 
\[
\mL_\theta = \mQ_\theta \mR_\theta
\]
is the lean QR factorization of $\mL_\theta$, then the pseudoinverse of $\mL_\theta$ is
\[
\mL_\theta^\dagger = \mR_\theta^{-1} \mQ_\theta^\mT.
\]
For small  and midsize problems the lean QR algorithm may be practical, however for larger problems this step may turn out to become a computational bottleneck. As pointed out in \cite{pragliola2022overcomplete}, the computation of the QR factorization can be avoided, because it is not the matrix that is needed, but only the matrix-vector product with the pseudoinverse and its transpose.

In fact, the product $z = \mL_\theta^\dagger \xi$ can be computed by solving the normal equations
\[
 \big(\mL_\theta^\mT \mL_\theta \big) z  = \mL_\theta^\mT \xi,
\]
and the product  $\zeta = \big(\mL_\theta^\dagger\big)^\mT y$ can be computed by solving
\[
  \big(\mL_\theta^\mT \mL_\theta\big) z  = y, \quad \zeta = \mL_\theta z,
\]
taking advantage of the sparsity of the matrix $\mL$.
As demonstrated in the computed examples in the next section, this implicit way of applying the pseudoinverse may  have a significant impact on the computing times.  

\section{Computed examples}

In this section, we consider two computed examples. The first one discusses the selection of the Tikhonov regularization parameter via the Bayesian interpretation, and the second one a strongly underdetermined inverse problem with sparsity constraint.

\subsection{Tikhonov regularization reinterpreted}

To put the discussion of this work in context of Tikhonov regularization, we consider the case of trivial partitioning, that is, the parameter vector $\theta$ comprises one single component, $\theta\in\R_+$. We consider a linear model,
\[
 b = \mA x + \varepsilon, \quad \varepsilon\sim{\mathcal N}(0,\sigma^2 \mI_m).
\]
 We adopt here the definition of the signal-to-noise ratio as
 \[
  {\rm SNR} = \frac{{\mathbb E}(\|b\|^2)}{{\mathbb E}(\|\varepsilon\|^2)}.
 \] 
where ${\mathbb E}$ stands for expectation. In the classical non-Bayesian framework, in lack of better model, we may simply write
\[
 {\mathbb E}\big(\|b\|^2\big) = \|b\|^2, 
 \]
 and for the noise,
 \[
 \quad {\mathbb E}(\|\varepsilon\|^2) = {\rm trace}\big\{{\mathbb E}\big(\varepsilon \varepsilon^\mT\big) \big\}=  {\rm trace}\big\{
\sigma^2 \mI_m
 \big\} =
 m\sigma^2,
\] 
leading to the interpretation
\begin{equation}\label{SNR1}
 {\rm SNR} =  \frac{\|b\|^2}{m\sigma^2}.
\end{equation}
In the Bayesian context, on the other hand, we use the prior distribution to estimate the SNR. Since
\[
 X\mid\Theta \sim{\mathcal N}(0,\theta \mI_n),
\]
we write
\begin{eqnarray*}
 {\mathbb E}(\|\mA X\|^2\mid\theta) &=& {\rm trace}\big\{{\mathbb E}\big(\mA X X^\mT \mA^\mT \mid \theta\big)\big\} \\
 &=& {\rm trace}\big\{\theta\mA \mA^\mT \big\}  = \theta \|\mA\|_F^2,
\end{eqnarray*}
and assuming the hypermodel  $\Theta\sim{{\rm Gamma}(\beta,\vartheta})$, we obtain
\[
  {\mathbb E}(\|\mA X\|^2)  = \|\mA\|_F^2 {\mathbb E}(\Theta) = \beta\vartheta\|\mA\|_F^2.
\]  
This reasoning leads to the interpretation
\begin{equation}\label{SNR2}
 {\rm SNR} = \frac{\beta\vartheta \|\mA\|_F^2 + m\sigma^2}{m\sigma^2}.
\end{equation}
Therefore, if estimates of the noise variance $\sigma^2$ and the SNR are given, a reasonable choice for the scaling parameter value $\vartheta$ is
\begin{equation}\label{vartheta}
 \vartheta = \frac {m\sigma^2}{\beta\|\mA\|_F^2}({\rm SNR} - 1).
\end{equation}

To compare the results obtained by traditional Tikhonov regularization and the hierarchical models, we run the following test. Consider a mildly ill-posed problem of numerical differentiation of a function $u:[0,1]\to\R$ satisfying $u(0) = 0$. The data comprises noisy observations of the function at regular grid points,
\[
 b_j = u(t_j) + \varepsilon_j, \quad t_j = \frac jn,\quad 1\leq j\leq n,
\]
and we seek to estimate the values $x_j = u'(t_j)$ at the same points. To interpret the problem as an inverse problem, we write
\[
 u(t_j) = \int_0^{t_j} u'(s)ds,
\]
and we discretize the problem by approximating
\[
 u(t_j) \approx \frac 1n\sum_{\ell=1}^j u'(s_\ell), \quad s_\ell = \frac\ell n,
\]
leading to an approximate discrete problem
\[
 b = \mA x + \varepsilon, \quad \mA = \frac 1n \left[\begin{array}{cccc} 1 & & & \\ 1 & 1 & & \\ \vdots & & \ddots & \\ 1 & 1 & \cdots & 1\end{array}\right]\in\R^{n\times n},
\]   
where $m=n$. 
To compare the two approaches, we run both the traditional Tikhonov regularization with Morozov discrepancy principle and the corresponding hierarchical model. To this end, we define the relative noise level $\sigma_{\rm rel}$, $0<\sigma_{\rm rel}<1$, and the absolute noise level by
\[
 \sigma^2 = \frac 1m\|b_0\|^2  \sigma_{\rm rel}^2, 
\]
where $b_0$ is the simulated noiseless signal. The simulated noisy signal is generated by setting
\[
 b = b_0 + \sigma w, \quad w\sim{\mathcal N}(0, \mI_m).
\]  
We then solve the inverse problem numerically using the Tikhonov regularization, setting the regularization parameter with Morozov discrepancy principle based on the condition
\begin{equation}\label{Morozov 2}
h(\alpha) =  \| \mA x_\alpha - b\|^2 = m\sigma^2 = {\mathbb E}\big(\|\varepsilon\|^2\big).
\end{equation}

To run the corresponding hierarchical model, we use the same data $b$, assuming that  the noise level $\sigma$ is known. To obtain an estimate for the SNR, we use simply formula (\ref{SNR1}) and substitute it in (\ref{vartheta}) to obtain the value of the scaling parameter $\vartheta$. In our example, the shape parameter $\beta$ is set to value $\beta = 3/2+\eta$ with $\eta =0.0001$. Observe that for low noise levels, we have $\|b\|^2 \approx \|b_0\|^2$, implying that formula (\ref{SNR1}) gives
\[
 {\rm SNR} \approx \frac{\|b_0\|^2}{m\sigma^2} = \frac 1{\sigma_{\rm rel}^2},
\]
however, we do not use this observation in our calculations. 

In our numerical experiment, we consider numerical differentiation of the function
\[
 u(t) = 1 + {\rm erf}(6t + 3), \quad 0<t<1,
\]
where ${\rm erf}$ denotes the error function. Observe that $u(0) = 1 + {\rm erf}(-3) \approx 2.2\times 10^{-5}$, so the assumption $u(0) = 0$ is valid with sufficient precision.  The analytic  solution of the differentiation problem is
\[
 u'(t) = \frac{12}{\sqrt{\pi}}e^{-(6t-3)^2}.
\]
We discretize the problem using the regular grid $t_j = j/n$ with $n = 50$, and use the second order smoothness prior, with 
\[
 \mL = \left[\begin{array}{rrrrr} -2 & 1 & & & \\ 1 & -2 & 1 & & \\ &\ddots &\ddots & \ddots & \\ &&1 &-2 &1
 \\ &&&1 &-2 \end{array}\right]\in\R^{n\times n}.
\]  
In our simulation, we pick 30 logarithmically uniformly distributed values of the relative noise level parameter $  \sigma_{\rm rel} \in[0.001,0.1]$ and solve the Tikhonov regularized solution using both the Morosov discrepacy principle and the IAS algorithm to select the regularization. 
To solve $\alpha$ numerically from the condition (\ref{Morozov 2}), we use the geometric bisection principle, choosing the new value of the regularization parameter in the current interval $[\alpha_{\rm min},\alpha_{\rm max}]$, where $h(\alpha_{\rm min})<m\sigma^2 \leq h(\alpha_{\rm max})$,  by
\[
 \alpha = \sqrt{\alpha_{\rm min} \alpha_{\rm max}},
\]
stopping the iterations when $|\alpha - \sqrt{m}\sigma|< 0.01\,m \sqrt{\sigma}$. The initial interval is set as $\alpha_{\rm min} = 10^{-16}$ and $\alpha_{\rm max} = 10^{10}$. We denote by $\alpha_{\rm Tikh}$ the regularization parameter obtained by the Morozov discrepancy principle.

The IAS iteration is initiated at value $\theta = \vartheta$, and the stopping criterion is set as $|\theta^j - \theta^{j-1}|/|\theta^{j-1}<0.01$.  When the iterations terminate, we define the corresponding regularization parameter as
\[
 \alpha_{\rm IAS} = \frac{\sigma}{\sqrt{\theta}}.
\]  

\begin{figure}[ht!]
\centerline{\includegraphics[width=15cm]{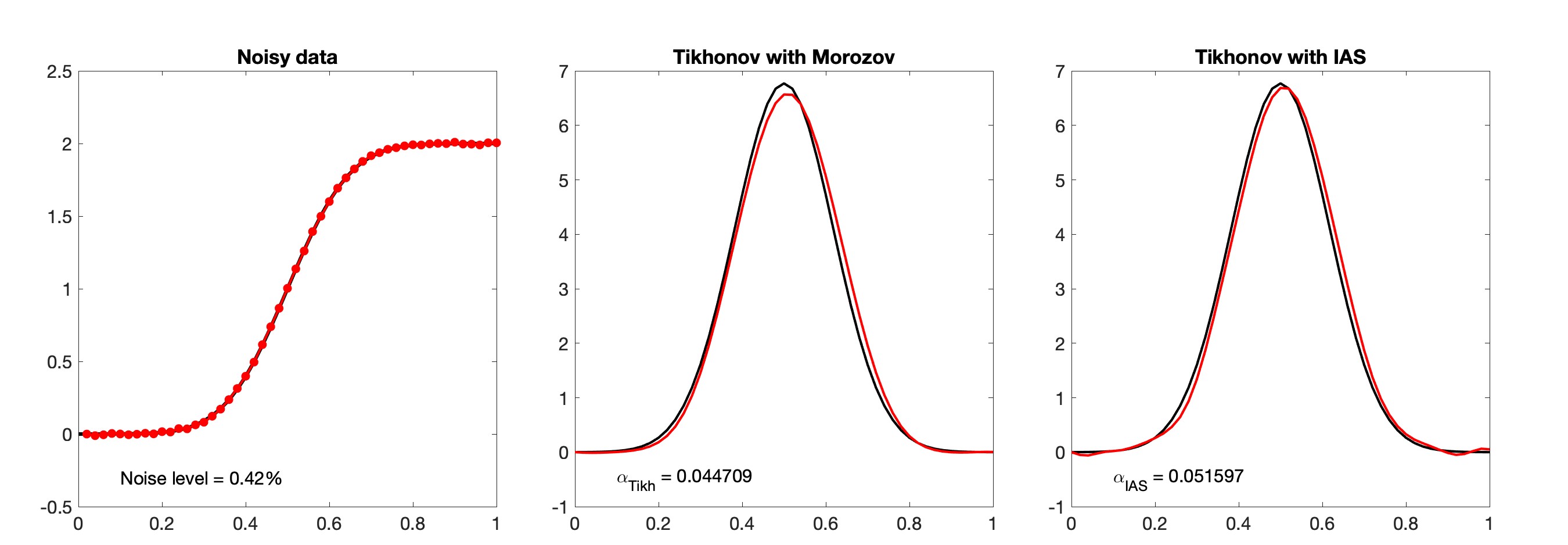}
}
\centerline{\includegraphics[width=15cm]{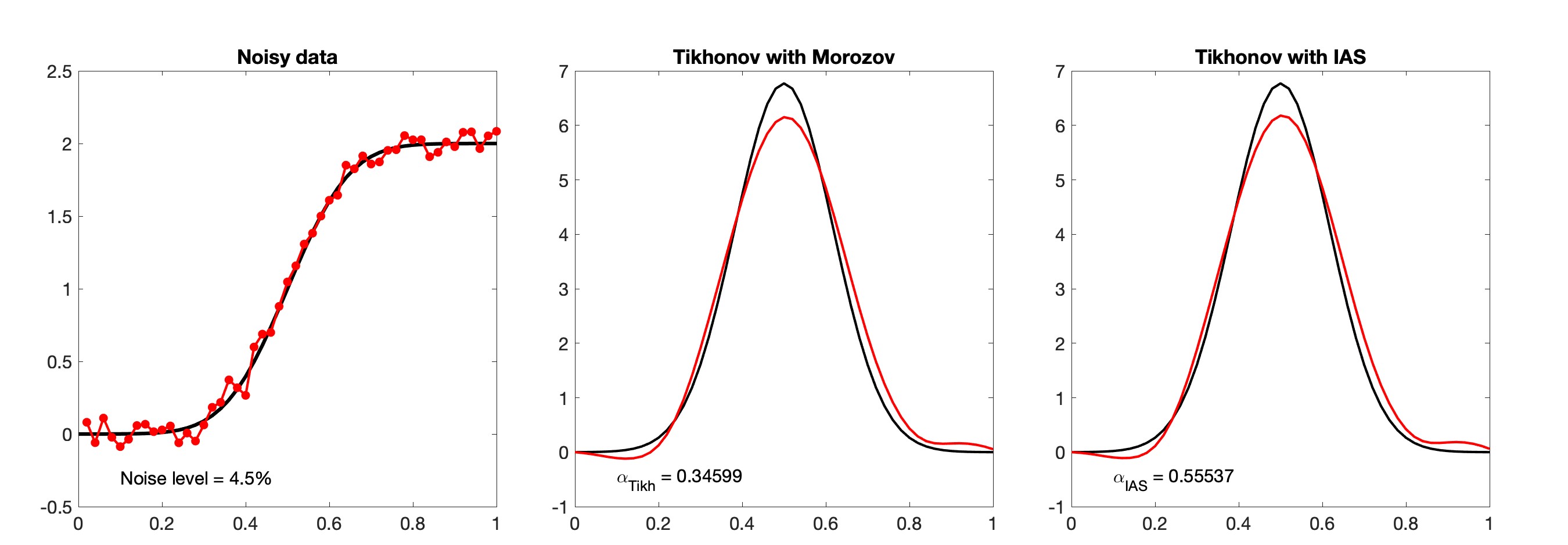}
}
\caption{\label{fig:Tikh results} Tikhonov-regularized solutions of the numerical differentiation using the two different methods for choosing the regularization parameter. The left panels show the noisy data, the middle panels the solutions based on Morozov discrepancy principle, and on the left, the solutions using the IAS-based criterion for the regularization parameter.}
\end{figure}

\begin{figure}[ht!]
\centerline{\includegraphics[width=17cm]{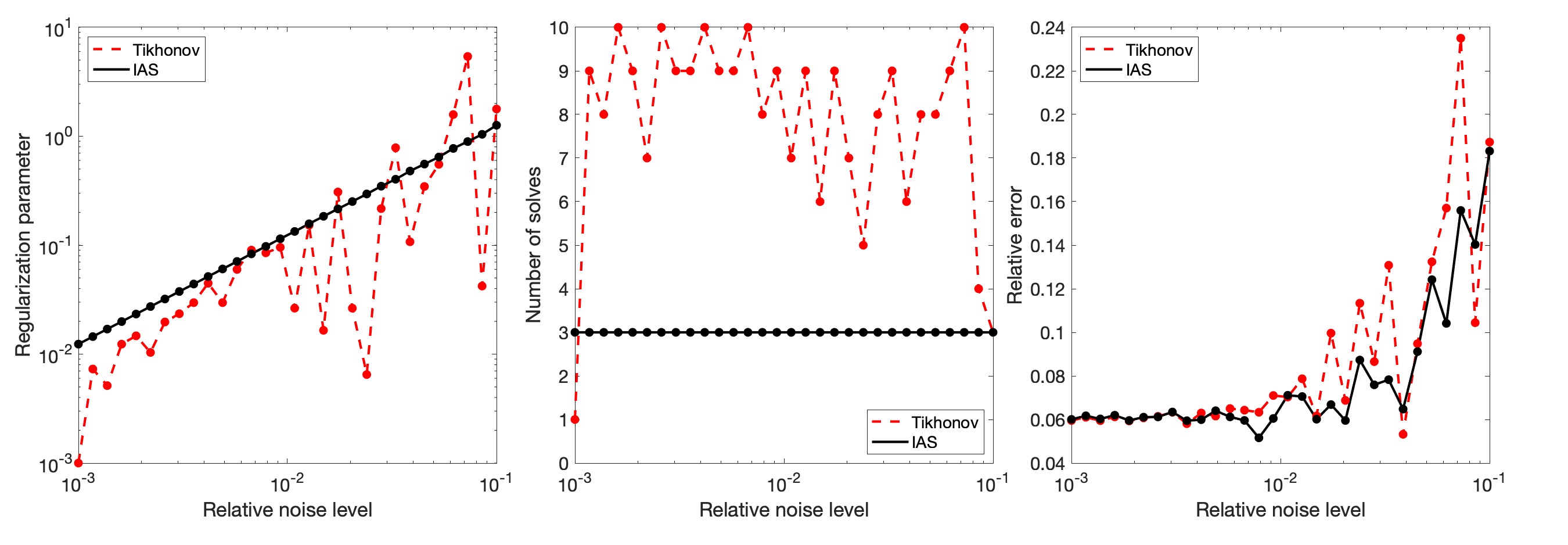}
}
\caption{\label{fig:Tikh summary} The left panel shows the logarithmic plot of the regularization parameter values using the two selection criteria as a function of relative noise level. The middle panel shows the number of times that the linear system needs to be solved for estimating the parameter, and on the right, the relative reconstruction error using the two methods are shown.}
\end{figure}

Figure~\ref{fig:Tikh results} shows two arbitrarily chosen noise level results using both selection criteria. We see that the regularization parameter values are in both cases comparable, and the computed estimates have no significant difference. In Figure~\ref{fig:Tikh summary}, the parameters are plotted as a function of the relative noise level. We observe that $\alpha_{\rm IAS}$ grows monotonically as a function of $\sigma_{\rm rel}$ following a clear power law, while the behavior of the parameter $\alpha_{\rm Tikh}$ is less predictable. Surprisingly, the growth is not even monotonic.  The middle panel of the same figure shows the number of solutions of the linear system of equations need to reach the prescribed convergence. While in both cases, the number is moderate, the IAS algorithm requires consistently only three iterations regardless of the noise level, while the Morozov discrepancy principle requires up to ten iterations. Again, the result for the traditional method is less predictable. Finally, we consider the performance of the two methods in terms of the relative error, $\|x_\alpha -x_{\rm true}\|/\|x_{\rm true}\|$. Both regularization methods perform in a comparable manner, the IAS estimate giving slightly lower error estimates than the discrepancy-based solution.

\subsection{Sparse reconstructions}

In the following example, we consider a fan beam tomography problem in two dimensions. A circular object $\Omega\subset\R^2$ is illuminated by a point-like X-ray source sending a fan of rays through the object. On the opposite side of the target, the attenuated intensities of the rays are registered. The measurements are repeated with several transmitter-receiver positions obtained by rotating the projection angle, as shown schematically in Figure~\ref{fig:tomoconfig}. The goal is to reconstruct the density distribution $\rho$ of the object $\Omega$, assuming that the attenuation of a ray traversing the object follows the Beer-Lambert law,
\[
 dI = \mbox{attenuation over $[s,s+ds]$} = - I (s)\rho\big(x(s)\big) ds,
\]
where $x = x(s)$ is the arc length parametrization of a given ray. By integrating the above formula over the ray, measured intensity  $I$ of the ray of length $L$ at the receiver satisfies
\[
   \log\frac{I}{I_0} = -\int_0^L \rho\big(x(s)\big) ds, \quad \mbox{$I_0$ = intensity of the cource,}
\]
hence the data consist of integrals of the unknown density over a set of rays. We discretize the density by approximating it by a piecewise linear function over a triangularization of $\Omega$ (see Figure~\ref{fig:tomoconfig}), the degrees of freedom being the nodal values of the density, $x_j = \rho(v_j)$, $1\leq j\leq n$.

\begin{figure}[ht!]
\centerline{
\includegraphics[width=7.5cm]{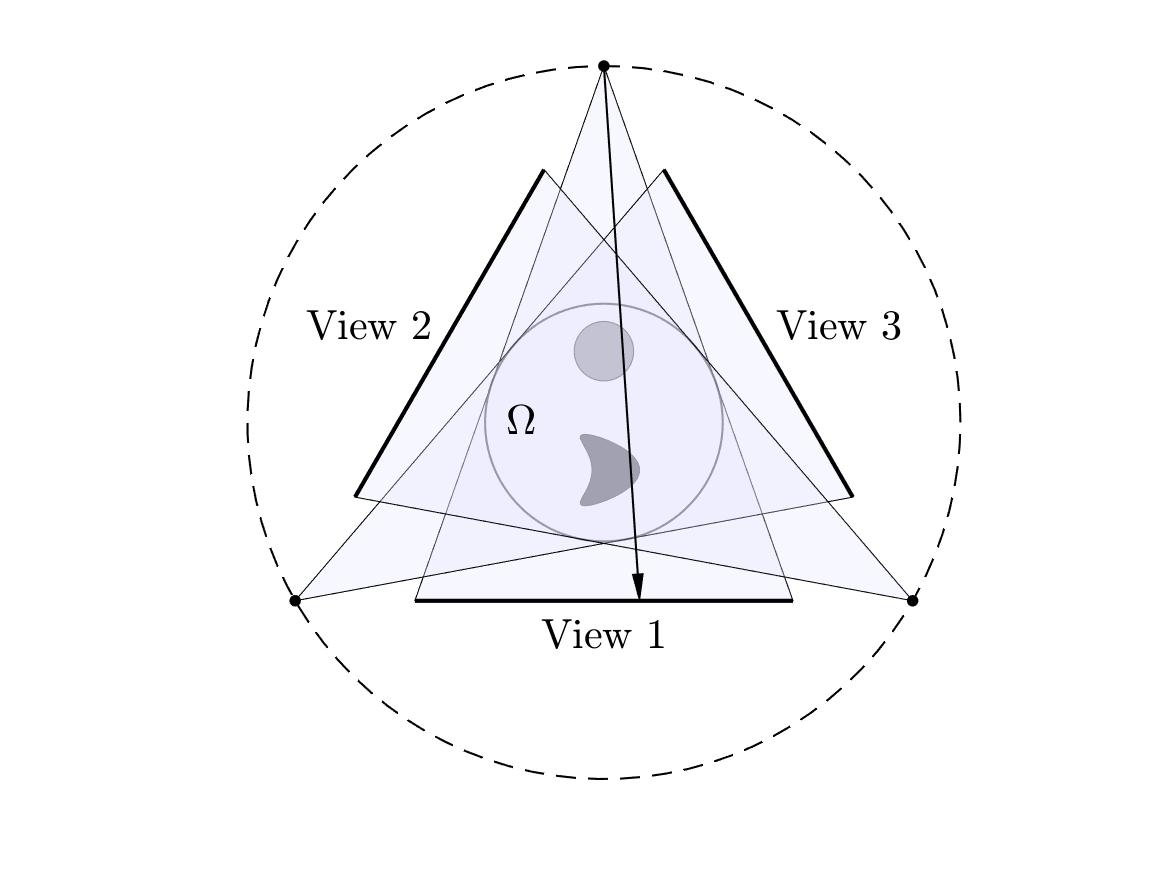}
\includegraphics[width=7.5cm]{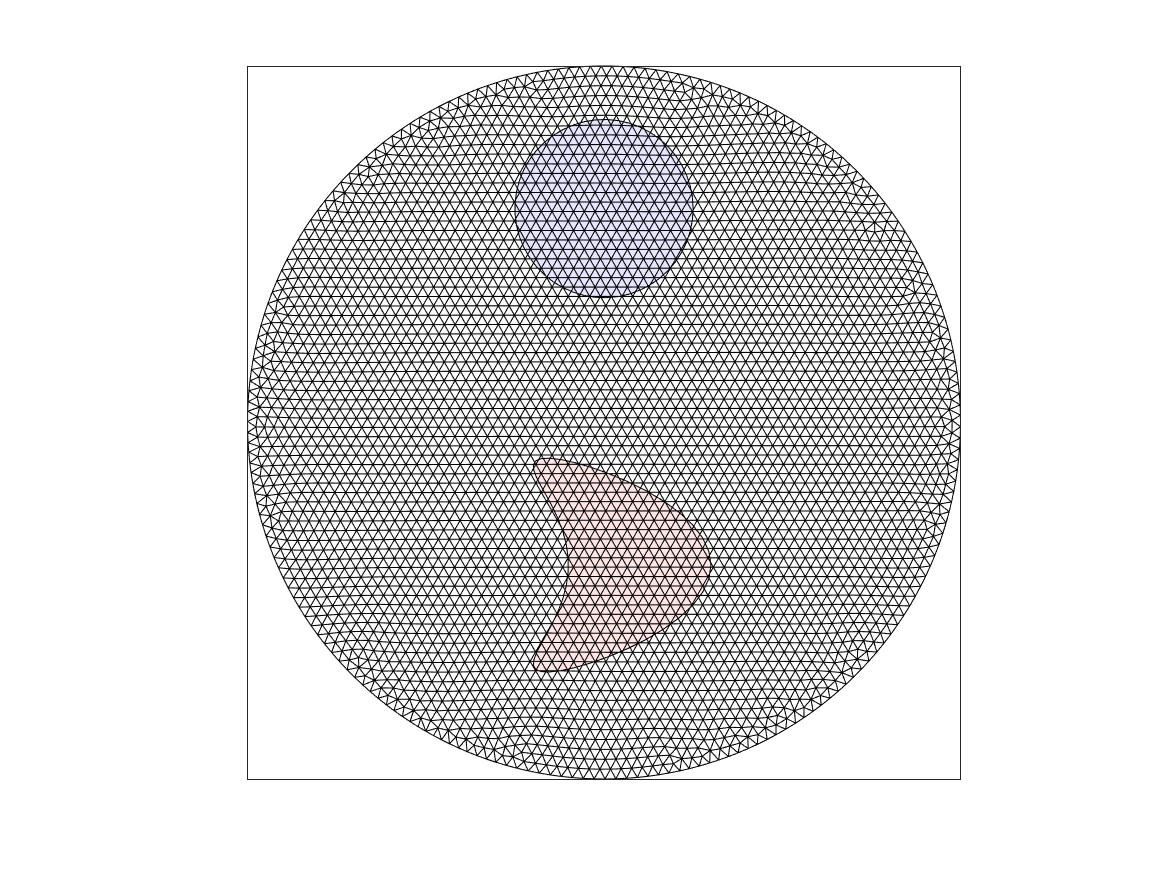}
}
\caption{\label{fig:tomoconfig} The geometric configuration of the fan beam tomography problem is shown schematically on the left, where three projection views are included. In the panel on the right, the true target is superimposed on the discretization mesh used in the inverse problem. The true densities of the inclusions are $\rho = 1.2$ in the pink kite shaped inclusion and $\rho=1$ in the blue  circular inclusion.
The number of interior nodes in this mesh is $n = 3\,821$, which is the number of unknowns in the inverse problem.}
\end{figure}

In this problem, we implement a first order smoothness prior matrix $\mL$ such that if $x$ contains the nodal values of the density, the vector $y = \mL x$ contains the increments along each edge in the mesh. Thus, each row of $\mL$ contains two non-zero values $\pm 1$. To guarantee that $\mL$ is full rank, we set the nodal values to zero, and delete the rows corresponding to boundary edges and columns corresponding to boundary nodes. This way, $\mL\in\R^{k\times n}$, and the number of degrees of freedom in the inverse problem is $n = 3\,821$, and the number of non-boundary edges is $k = 11\,872$, which is the number of unknowns in the IAS algorithm.

We generate a target density distribution consisting of two inclusions of constant density as shown in Figure~\ref{fig:tomoconfig}. The integrals along the rays in this case can be computed exactly, as the boundary curves of the inclusions are known and the densities are constant. We assume that each projection comprises $m_1 = 300$ rays, and $m_2 = 15$ equally distributed projection angles are included, implying that the dimension of the data vector $b$ is $m = m1\times m_2 = 1\,500$. Finally, we corrupt the data by additive Gaussian scaled white noise with noise level equal to 1\% of the maximum of the noiseless data.

We then run the IAS algorithm with the hyperparameter choices  $\eta = 0.001$ and  $\vartheta_j = 0.05$. In the tomography problem, the sensitivity of the data to different components does not change dramatically, so including the detailed sensitivity analysis is not as critical as in problems where ,e.g., distance to receivers count.  
Figure~\ref{fig:tomoresults} shows the density reconstructions of the six first iterations. After six iterations, the reconstructions do not change significantly.

\begin{figure}[ht!]
\centerline{
\includegraphics[width=13cm]{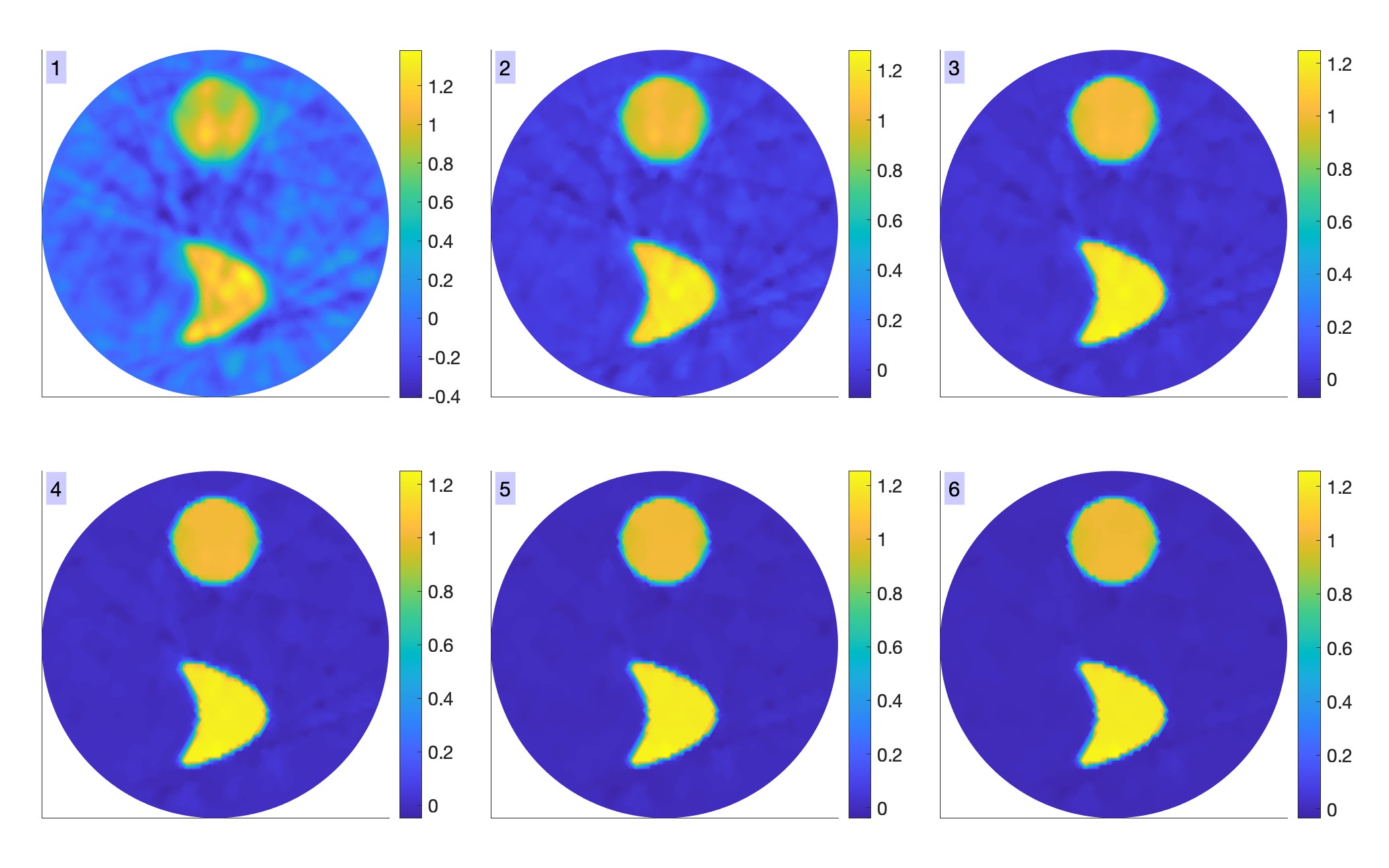}}
\centerline{
\includegraphics[width=13cm]{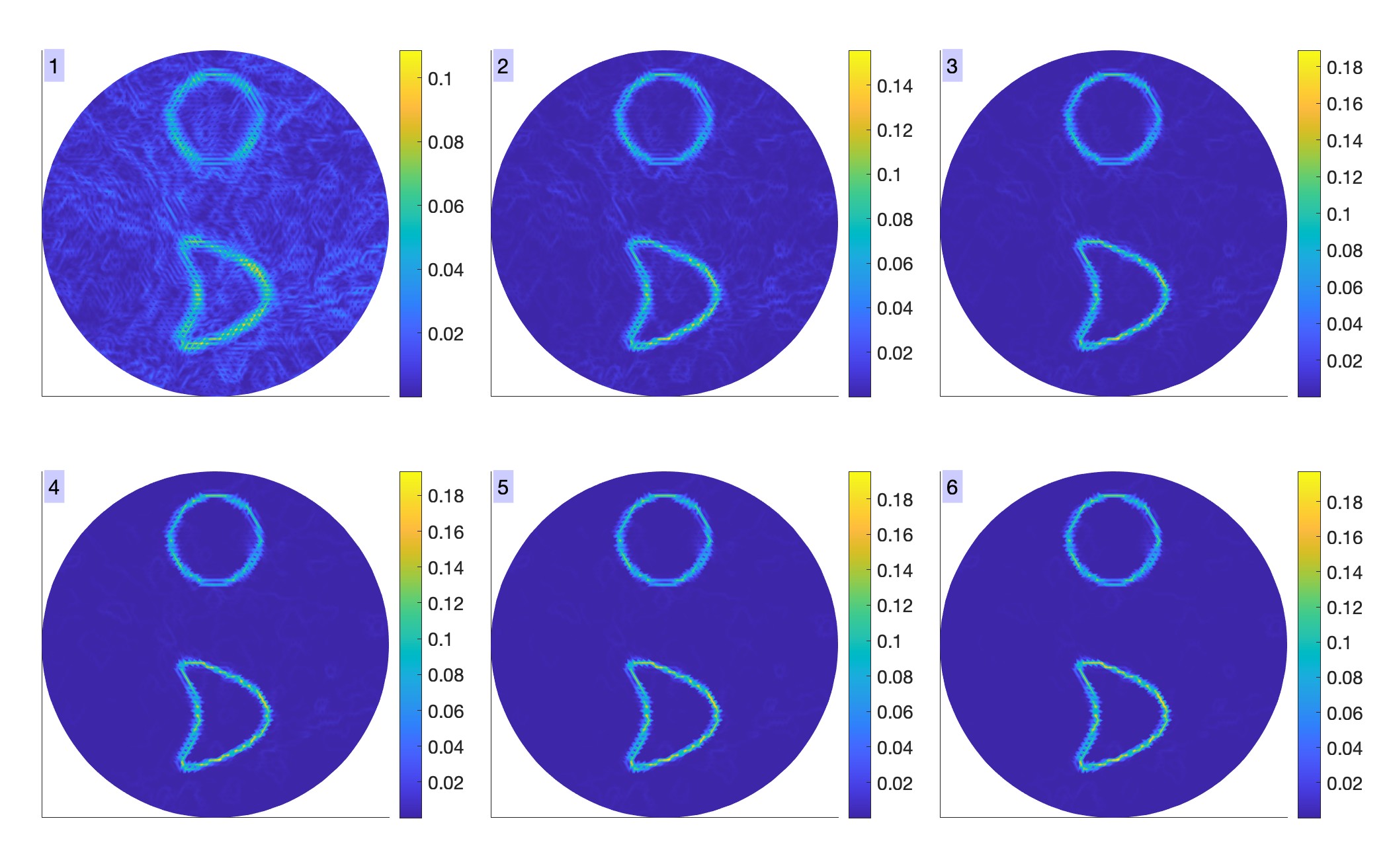}}
\caption{\label{fig:tomoresults} Six first iterations of the density reconstruction of the IAS algorithm (two first rows, lexicographical order), and the corresponding variances $\theta$ (two last rows, lexicographical order). After six iterations, the reconstructions remain visually very similar to the last reconstruction.}
\end{figure}

To demonstrate the importance of avoiding the QR factorization of the matrix $\mL_\theta$, we run the algorithm using the lean QR algorithm, comparing the computing times with those of the algorithm that avoids the factorization as detailed in section~\ref{sec:qr}. The computing times are shown in Figure~\ref{fig:times}. We observe that the time needed for each iteration without the QR factorization is two orders of magnitude smaller than the time with the factorization. It is to be expected that with larger problems, the savings become even more significant. The reconstructions (not shown here) are visually indistinguishable from the ones without the decomposition.

\begin{figure}[ht!]
\centerline{
\includegraphics[width=17cm]{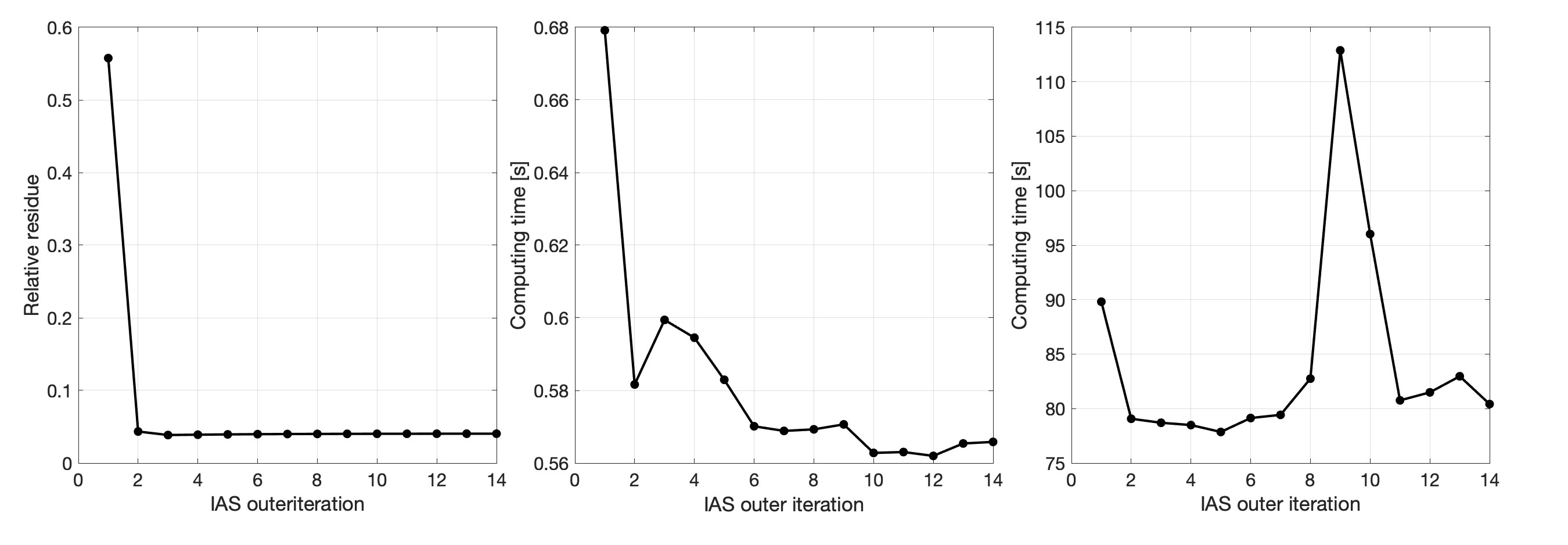}
}
\caption{\label{fig:times} The left panel shows the relative residues $\|b - \mA\xi\|/\|b\|$ as the iterations avoiding the QR decomposition progresses. The computing times per outer iterations of the QR-free iterations are plotted in the middle panel. For comparison, we ran the algorithm using the QR-decomposition at each outer iteration, the computing times being shown in the right panel.
}
\end{figure}

\section{Conclusions}

The formal similarity of Tikhonov-type regularization methods and the maximum a posteriori estimation methods has been known for long, however, a clear connection between the popular methods for setting the regularization parameter, including the discrepancy principle, GCV, and UPRE in the one hand, and Bayesian methods on the other hand is less well known. Comparisons between the classical methods and selection criteria based on Bayesian methods such as the maximum evidence can be found in the literature, see, e.g., \cite{hall2023efficient,pereira2015empirical,chung2017generalized,saibaba2020efficient,huri2016selecting,bardsley2011techniques}. One of the aims of this article is to make the connection clearer, basing the parameter selection on a unified MAP estimation process for both the unknown and the regularization parameter. One of the interesting findings here is the striking power law between the regularization  parameter $\alpha$ and the relative noise level, as shown in Figure~\ref{fig:Tikh summary}. Such dependency is often postulated in theoretical  articles in which Bernstein-von Mises type posterior convergence results are studied: 
In order to get convergence of the posterior densities in the small noise limit, the authors  ``make the regularization disappear in a carefully chosen way" \cite{agapiou2013posterior,agapiou2014bayesian}. The hierarchical Bayesian models may provide a natural setting in which such requirement follows automatically.

The use of Lanczos bidiagonalization to solve the linear system adjoint to the normal equations arising in the IAS algorithm puts the computation on more solid ground, as the need of resorting to approximate methods such as replacing the standard Tikhonov type problem by, e.g., CGLS approximation with early stopping  (see, e.g., \cite{calvetti2023bayesian} )may not be necessary. The computing times shown in Figure~\ref{fig:times} support this view.

\section*{Acknowledgements}

The work of DC was partly supported by the NSF grant DMS 1951446, and that of  ES by the NSF grant DMS-2204618.

{\bf Conflict of interest:} The authors have no competing interests to declare that are relevant to the content of this article.

{\bf Data availability:} No datasets were generated or analyzed during the current study.

\bibliographystyle{siam} 
\bibliography{biblio_Tikh}
\end{document}